\documentclass{amsart}
\usepackage[utf8]{inputenc}
\usepackage{tikz} \usetikzlibrary{arrows,shapes,trees}

\title[Goldman-Turaev formality from the KZ connection]{Goldman-Turaev formality from the Knizhnik-Zamolodchikov connection}

\author{Anton Alekseev}
\address{Department of Mathematics, University of Geneva, 2-4 rue du Li\`evre, c.p. 64, 1211 Geneva, Switzerland}
\email{Anton.Alekseev@unige.ch}
\author{Florian Naef}
\address{Department of Mathematics, University of Geneva, 2-4 rue du Li\`evre, c.p. 64, 1211 Geneva, Switzerland}
\email{Florian.Naef@unige.ch}

\DeclareMathOperator{\Tr}{Tr}
\DeclareMathOperator{\dlog}{dlog}
\DeclareMathOperator{\Hol}{Hol}
\newcommand{\ot}{\leftarrow}

\newtheorem{Thm}{Theorem}[section]
\newtheorem{Prop}[Thm]{Proposition}
\newtheorem{Lem}[Thm]{Lemma}
\newtheorem{Rem}[Thm]{Remark}

\begin{document}
\begin{abstract}
For an oriented 2-dimensional manifold $\Sigma$ of genus $g$ with $n$ boundary components the space $\mathbb{C}\pi_1(\Sigma)/[\mathbb{C}\pi_1(\Sigma), \mathbb{C}\pi_1(\Sigma)]$ carries the Goldman-Turaev Lie bialgebra structure defined in terms of intersections and self-inter\-sections of curves. Its associated graded Lie bialgebra (under the natural filtration) is described by cyclic words in $H_1(\Sigma)$ and carries the structure of a necklace Schedler Lie bialgebra. The isomorphism between these two structures in genus zero has been established in \cite{mas} using Kontsevich integrals and in \cite{genus0} using solutions of the Kashiwara-Vergne problem.

In this note we give an elementary proof of this isomorphism over $\mathbb{C}$. It uses the Knizhnik–Zamolodchikov connection on $\mathbb{C}\backslash\{ z_1, \dots z_n\}$. We show that the isomorphism naturally depends on the complex structure on the surface. 
The proof of the isomorphism for Lie brackets is a version of the classical result by Hitchin \cite{hitchin}. Surprisingly, it turns out that a similar proof applies to cobrackets.

Furthermore, we show that the formality isomorphism constructed in this note coincides with the one defined in \cite{genus0} if one uses the solution of the Kashiwara-Vergne problem corresponding to the Knizhnik-Zamolodchikov associator.
\end{abstract}

\maketitle

\section{Holonomy maps}

In this section we recall the definition of the Knizhnik-Zamolodchikov connection and define the holonomy map using iterated integrals.

\subsection{The Knizhnik-Zamolodchikov connection}
For $n \in \mathbb{Z}_{\geq 1}$ let $\mathfrak{t}_{n+1}$ be the Drinfeld-Kohno Lie algebra of infinitesimal braids. It has generators $t_{ij}=t_{ji}$ with $i \neq j$ for $i,j=1, \dots, n+1$ and relations $[t_{ij}, t_{kl}]=0$ for all quadruples of distinct labels $i,j,k,l$ and $[t_{ij}+t_{ik}, t_{jk}]=0$ for all triples $i,j,k$. The Lie subalgebra of $\mathfrak{t}_{n+1}$ generated by $a_i=t_{i(n+1)}$ for $i=1, \dots, n$ is a free Lie algebra with $n$ generators.

There is a canonical flat Knizhnik-Zamolodchikov (KZ) connection on the configuration space ${\rm Conf}_{n+1}(\mathbb{C})$ of $n+1$ points
$z_1, \dots, z_{n+1} \in \mathbb{C}$ with values in $\mathfrak{t}_{n+1}$:
$$
d+A_{\rm KZ} = d+ \frac{1}{2\pi i} \, \sum_{i<j} t_{ij} \, d\log(z_i-z_j).
$$
The fiber of the forgetful map ${\rm Conf}_{n+1}(\mathbb{C}) \to {\rm Conf}_n(\mathbb{C})$ (by forgetting the point $z=z_{n+1}$) over $(z_1, \dots, z_n) \in {\rm Conf}_n(\mathbb{C})$ is given by $\Sigma = \mathbb{C} \setminus \{ z_1, \dots, z_n \}$. The restriction of the KZ connection to the fiber is of the form:
\begin{equation} \label{connection}
d+A = d + \frac{1}{2\pi i} \, \sum_{i=1}^n a_i \, d\log(z-z_i).
\end{equation}
Note that the 1-forms $\frac{1}{2\pi i} \, d\log(z-z_i)$ form a basis in the cohomology  $H^1(\Sigma)$. Then, we can naturally view $a_i$'s as a basis of $H_1=H_1(\Sigma)$. In fact, this is a natural basis of $H_1$ given by the cycles around the marked points.
It is convenient to use the notation
$$
A(z) = \frac{1}{2\pi i} \sum_{i=1}^n a_i d\log(z -z_i) \in \Omega^1(\Sigma) \otimes H_1.
$$

\subsection{Free algebras and iterated integrals}
The degree completed Hopf algebra $TH_1$ over $\mathbb{C}$ is naturally isomorphic to the completed universal enveloping algebra of the free Lie algebra $\mathbb{L}(a_1, \dots, a_n)$ with generators $a_1, \dots, a_n$. Group-like elements in $TH_1$ form a group $G_n=\exp(\mathbb{L}(a_1, \dots, a_n))$ isomorphic to $\mathbb{L}(a_1, \dots, a_n)$ equipped with the group law defined by the Baker-Campbell-Hausdorff series.

For a path $\gamma$ parametrized by  $z: [0,1] \to \Sigma$, we define the holonomy of the connection $d+A$ using interated integrals:
$$
\begin{array}{lllll}
{\rm Hol}_\gamma & = & \sum_{k=0}^\infty \int_{1\geq t_1 \geq \dots \geq t_k \geq 0} A(z(t_1)) \dots A(z(t_k)) & & \\
& = & \sum_{\mathbf{i}} a_{i_1} \dots a_{i_k} \, \int_\gamma \frac{dz(t_1)}{z(t_1)-z_{i_1}} \circ \dots 
\circ \frac{dz(t_k)}{z(t_k)-z_{i_k}} & \in & G_n,
\end{array}
$$
where $\mathbf{i}=(i_1, \dots, i_k)$.
Since the connection $d+A$ is flat, the holonomy ${\rm Hol}_\gamma$ is independent of the homotopy deformations of $\gamma$ with fixed end points. Hence, for small deformations of $\gamma$ one can view ${\rm Hol}_\gamma$ as a functions of its endpoints. The de Rham differential of this function is given by 
$$
d \, {\rm Hol}_\gamma = A(z(1)) \, {\rm Hol}_\gamma - {\rm Hol}_\gamma \, A(z(0)).
$$
Let $\gamma$ be a closed path and consider the path $\gamma(s)$ which starts at $z(s)$, follows $\gamma$ and ends at $z(s)$. The corresponding holonomy is denoted by ${\rm Hol}_\gamma(s) = {\rm Hol}_{\gamma(s)}$ and its de Rham differential has the form
\begin{equation}  \label{Hol(s)}
d \, {\rm Hol}_\gamma(s) = A(z(s))\, {\rm Hol}_\gamma - {\rm Hol}_\gamma \, A(z(s)).
\end{equation}
For a closed path $\gamma$ and two points $s \neq t \in S^1$ we denote by $\gamma(t \ot s)$ the oriented path starting at $z(s)$, following $\gamma$ and ending at $z(t)$. The corresponding holonomy is denoted ${\rm Hol}_\gamma(t \ot s)={\rm Hol}_{\gamma(t \ot s)}$ and its de Rham differential is given by
$$
d \, {\rm Hol}_\gamma(t \ot s) = A(z(t)) \, {\rm Hol}_\gamma(t \ot s) -  {\rm Hol}_\gamma(t \ot s) \, A(z(s)).
$$

Choose a base point $p\in \Sigma$ and denote by $\pi_1=\pi_1(\Sigma, p)$ the fundamental group of $\Sigma$ with base point $p$. Let $\gamma$ be a closed path based at $p$. The map
$$
W: \gamma \mapsto {\rm Hol}_\gamma
$$
descends to a group homomorphism $\pi_1 \to G_n$ and induces an isomorphism of completed Hopf algebras $\mathbb{C}\pi_1 \to TH_1$.

Denote by 
$$
|\mathbb{C} \pi_1| = \mathbb{C} \pi_1/[\mathbb{C} \pi_1,\mathbb{C} \pi_1]
$$
the space spanned by conjugacy classes in $\pi_1$ and similarly 
$$
|TH_1| = TH_1/[TH_1, TH_1]
$$ 
the space spanned by cyclic words in $H_1$. Note that $|\mathbb{C} \pi_1|$ is also isomorphic to the  vector space spanned by free homotopy classes of loops in $\Sigma$. The  map $W$ induces a map $|W|: |\mathbb{C} \pi_1| \to |TH_1|$ which is independent of the base point $p$. In what follows we will study properties of this map.

\begin{Rem}
For higher genus surfaces, iterated integrals of the harmonic Magnus connection were studied in \cite{Kawazumi}. It is not known whether the harmonic Magnus expansion gives rise to a Goldman-Turaev formality map.
\end{Rem}

\section{Goldman-Turaev formality}

The space $|\mathbb{C} \pi_1|$ carries the canonical Goldman-Turaev Lie bialgebra structure which depends on the framing of $\Sigma$. Moreover, it is canonically filtered by powers of the augmentation ideal of the group algebra $\mathbb{C}\pi_1$. The space $|TH_1|$ carries the necklace Schedler Lie bialgebra structure (see \cite{Schedler}) which depends on the choice of a basis in $H_1$. The main result of this note is an elementary proof of the following theorem:

\begin{Thm} \label{main}
The map $|W|: |\mathbb{C} \pi_1| \to |TH_1|$ is an isomorphism of Lie bialgebras for the Goldman-Turaev Lie bialgebra structure on (completed) $|\mathbb{C}\pi_1|$ defined by the blackboard framing and the necklace Schedler Lie bialgebra structure on $|TH_1|$ defined by the natural basis $\{ a_1, \dots, a_n\} \subset H_1$. 
\end{Thm}


\subsection{Kirillov-Kostant-Souriau double bracket}

Recall the Kirillov-Kostant-Souriau (KKS) double bracket 
(in the sense of van den Bergh \cite{vdb}) on $TH_1$ which is completely determined by its values on generators:
$$
a_i \otimes a_j \mapsto \{a_i, a_j\} = \delta_{ij} (1 \otimes a_i - a_i \otimes 1) =
\delta_{ij} \, (1 \wedge a_i). 
$$
One of the key observations is the following lemma (reinterpreting \cite{fockrosly}):
\begin{Lem}  \label{bracketA}
\begin{align}
\{A(z), A(w) \} 
&= \frac{1}{2\pi i}(1 \wedge (A(z)-A(w)) )\, d\log(z-w) \nonumber
\end{align}
\end{Lem}

\begin{proof}
The proof is by the direct computation:
\begin{align*}
\{ A(z), A(w)\} & = \frac{1}{(2\pi i)^2} \, \{ \sum_i a_i d\log(z-z_i), \sum_j a_j d\log(w-z_j) \} \\
& =  \frac{1}{(2\pi i)^2} \, \sum_i (1 \wedge a_i) d\log(z-z_i) d\log(w-z_i) \\
& =  \frac{1}{(2\pi i)^2} \, \sum_i (1 \wedge a_i) (d\log(z-z_i)- d\log(w-z_i)) d\log(z-w) \\
& =  \frac{1}{2\pi i}(1 \wedge (A(z)-A(w))) d\log(z-w).
\end{align*}
\end{proof}

\subsection{Variations of the holonomy map}
In order to proceed we will need some notation from non-commutative differential geometry. Let $\partial_i: TH_1 \to TH_1 \otimes TH_1$ for $i=1, \dots, n$ denote double derivations with the property $\partial_i a_j = \delta_{ij} (1 \otimes 1)$. They induce maps (denoted by the same letter) $\partial_i: |TH_1| \to TH_1$, and by composition
$\partial_i \partial_j: |TH_1| \to TH_1 \otimes TH_1$. 
In what follows we will use the formulas for the first and second derivatives of the elements
 $|W_\gamma|$ for $\gamma$ a closed path parametrized by a map $z: S^1 \to \Sigma$:
\begin{equation} \label{derivatives}
\begin{aligned}  
\partial_j |W_\gamma| & = \int_{S^1} \, {\rm Hol}_\gamma(s) \, d\log(z(s) - z_j), \\
\partial_i \partial_j |W_\gamma| & = \int_{S_1 \times S_1} \left( {\rm Hol}_\gamma(t \ot s) \otimes {\rm Hol}_\gamma(s \ot t)\right)  \, d\log(z(s) -z_i)d \log(z(t)-z_j),
\end{aligned}
\end{equation}
where in the expression for the second derivative the torus $S^1 \times S^1$ is oriented by the volume form $dsdt$.

\subsection{Necklace Schedler Lie bialgebra}
We will also make use of the maps 
$\Tr: TH_1 \otimes TH_1 \to |TH_1|$ given by $\Tr(a \otimes b)=|ab|$ and 
$\Tr^{12}: TH_1 \otimes TH_1 \to |TH_1| \otimes |TH_1|$ which is the component-wise projection
$\Tr^{12}(a\otimes b) = |a| \otimes |b|$. In terms of the KKS double bracket, the necklace Schedler Lie bialgebra structure looks as follows:
for $\psi, \psi' \in |TH_1|$ we have
\begin{align*}
[\psi, \psi'] &= \sum_{ij} \Tr \left( (\partial_i \psi\otimes \partial_j \psi') \{ a_i, a_j \} ) \right) \\
\delta \psi &= \sum_{ij} \Tr^{12} \left( (\partial_i \partial_j \psi) \{ a_i, a_j \} \right).
\end{align*}

\begin{Rem}
The necklace Lie algebra structure corresponding to the KKS double bracket first appeared in \cite{Drinfeld}. It was then generalized to other double brackets and other quivers (a quiver is a part of a general definition of a necklace Lie algebra) in \cite{Kontsevich, Ginzburg, Bocklandt}. The first description of the cobracket is in \cite{Schedler}. The formulas above represent the KKS necklace Schedler Lie bialgebra structure in a form convenient for our purposes.
\end{Rem}

\subsection{Proof of Theorem \ref{main}}
We are now ready to give a proof of Theorem \ref{main}:

\begin{proof}[Proof of the Lie homomorphism]

 Let $\gamma$ and $\gamma'$ be two closed loops in $\Sigma$ with a finite number of transverse intersections  parametrized by $z, w: [0,1] \to \Sigma$. Denote by $(s_i, t_i)$ pairs of parameters corresponding to intersection points, $p_i=z(s_i) = w(t_i)$ and by $D_i(\varepsilon) \subset S^1 \times S^1$ small discs of radius $\varepsilon$ around $(s_i,t_i)$ with boundaries small circles $S_i(\varepsilon)$ positively oriented under the orientation defined by the volume form $dsdt$ on the torus $S^1 \times S^1$. Let $\epsilon_i =+1$ if the orientation induced by the form $dsdt$ on $T_{p_i}\Sigma$ coincides with the blackboard orientation and $\epsilon_i=-1$ otherwise.

We compute the necklace Lie bracket of the elements $|W_\gamma|$ and $|W_{\gamma'}|$:
\begin{align*}
    [|W_\gamma|, |W_{\gamma'}|] &= \Tr\left( (\partial_i |W_\gamma| \otimes \partial_j |W_{\gamma'}|) \, \{ a_i, a_j\} \right) \\
    & = \int_{S^1 \times S^1} \Tr\left( (\Hol_\gamma(s) \otimes \Hol_{\gamma'}(t))  \{A(z(s)), A(w(t)) \} \right) \\
    &= \frac{1}{2\pi i}\int_{S^1 \times S^1} \Tr\Big( \big(\Hol_\gamma(s) \otimes \Hol_{\gamma'}(t)\big)  \big( 1 \wedge (A(z(s)) - A(w(t)) \big)  \Big)  \dlog(z-w) \\
    &=\lim_{\varepsilon \to 0} \frac{1}{2\pi i}\int_{S^1 \times S^1 \backslash \cup_i D_i(\varepsilon)} d |\Hol_\gamma(s) \Hol_{\gamma'}(t)| \dlog(z-w) \\
    &= - \sum_i |\Hol_\gamma(s_i) \Hol_{\gamma'}(t_i) | \, \lim_{\varepsilon \to 0} \,  \frac{1}{2\pi i} \, \int_{S_i(\varepsilon)}\,  d\,\log(z-w) \\
    &= \sum_i |\Hol_\gamma(s_i) \Hol_{\gamma'}(t_i) | \epsilon_i .
\end{align*}

Here in the first line we use the definition of the necklace Schedler Lie bracket, in the second line the expression \eqref{derivatives} for the first derivative of $|W_\gamma|$ and the definition of $A(z)$, in the third line Lemma \ref{bracketA},  in  the fourth line  formula \eqref{Hol(s)} for the differential of the holonomy map, in the fifth line the Stokes formula and in the sixth line the residue formula. Note that the result of the calculation is exactly the expression for the image under the map $W$ of the Goldman bracket \cite{Goldman} of the loops $\gamma$ and $\gamma'$.

\end{proof}

\begin{proof}[Proof of the cobracket homomorphism]

Next, consider a closed path $\gamma$ parametrized by $z: S^1 \to \Sigma$. Assume that $\gamma$ is an immersion with finitely many transverse self-intersections $p_i=z(s_i)=z(t_i)$ (it is convenient to have each point appear twice in the count with $s_j=t_i$ and $t_j=s_i$). Denote $\alpha^{\pm}(\varepsilon)$ the circles on $ S^1 \times S^1$ corresponding to the parameters $t=s \pm \varepsilon$ and by $\beta$ the strip between them. Similar to the proof above, let $D_i(\varepsilon)$ be small discs of radius $\varepsilon$ around the self-intersection points, $S_i(\varepsilon)$ their positively oriented boundaries and $\epsilon_i$ signs of self-intersections defined as above. Note that $\epsilon_j = - \epsilon_i$ for $i$ and $j$ representing the same self-intersection point.

We compute the necklace cobracket of the element $|W_\gamma|$:
\begin{align*}
    \delta |W_\gamma| &= \Tr^{12} \left( \partial_i \partial_j |W_\gamma| \, \{ a_i, a_j\} \right) \\
    & = \int_{S^1 \times S^1} \Tr^{12}(\Hol_\gamma(t \ot s) \otimes \Hol_\gamma(s \ot t)) \{A(z(s)), A(z(t)) \}) \\
    &= \lim_{\varepsilon \to 0} \frac{1}{2\pi i}\int_{S^1 \times S^1 \backslash \beta(\varepsilon) \cup (\cup_i D_i(\varepsilon))} d (\Tr^{12}(\Hol_\gamma(t \ot s) \otimes \Hol_\gamma(s \ot t)) \dlog(z(s)-z(t)) \\
 & = - \lim_{\varepsilon \to 0} \frac{1}{2 \pi i}\int_{ \cup_i S_i(\varepsilon) \cup \alpha^+(\varepsilon) \cup \alpha^-(\varepsilon) } (|\Hol(t \ot s)| \otimes |\Hol(s \ot t)|) \dlog(z(s)-z(t)) \\
 & = \sum_i |\Hol(t_i \ot s_i)| \otimes |\Hol(s_i \ot t_i)| \epsilon_i +  (1 \wedge |W_\gamma|) \frac{1}{2\pi i}\int_\gamma \, d \, \log(\dot{z}(t)) \\
&= 
\sum_i |\Hol(t_i \ot s_i)| \otimes |\Hol(s_i \ot t_i)| \epsilon_i + (|W_\gamma| \wedge |1|) \operatorname{rot}(\gamma).
\end{align*}
Here in the first line we use the definition of the necklace Schedler cobracket, in the second line the expression \eqref{derivatives} for the second derivative of $|W_\gamma|$, in the third line Lemma \ref{bracketA} and the expression \eqref{Hol(s)} for the de Rham differential of the holonomy map, in the fourth line the Stokes formula and in the fifth line the definition of the rotation number ${\rm rot}(\gamma)$. The resulting expression is exactly the Turaev cobracket of the path $\gamma$ with respect to the blackboard framing which defines the rotation  number.
\end{proof}

\begin{Rem}
In the proof above, the calculation of the Goldman bracket of the holonomies follows the standard scheme, see \cite{hitchin}, \cite{AM}. The calculation of the Turaev cobracket uses the same techniques, but it seems to be new.
\end{Rem}

\begin{Rem}
In \cite{Turaev}, the cobracket is defined on the space $|\mathbb{C}\pi_1|/\mathbb{C} |1|$, where $|1|$ is the homotopy class of the trivial loop on $\Sigma$. If one fixes a framing on $\Sigma$, this definition can be lifted  to $|\mathbb{C}\pi_1|$ by the formula above (for details, see \cite{Kawazumi2}).
\end{Rem}

\section{The $z$-dependence of the holonomy map}

In this section, we study the dependence of the holonomy map $|W|: |\mathbb{C} \pi_1| \to | TH_1| $ on the positions of the poles $z_1, \dots, z_n$ in the connection \eqref{connection}. We will use the notation $|W^{(z)}|$ to make this $z$-dependence more explicit.

\subsection{The bundle $\mathcal{G}_n$ }

Consider the natural fibration ${\rm Conf}_{n+1} \to {\rm Conf}_n$ and assign to a point of the base $(z_1, \dots. z_n)$ the Goldman-Turaev Lie algebra of the fiber $|\mathbb{C}\pi_1(\Sigma)|$ for $\Sigma=\mathbb{C}\backslash \{ z_1, \dots, z_n\}$. This defines a  bundle of Lie bialgebras $\mathcal{G}_n$ over ${\rm Conf}_n$. 

In a neighborhood of a point $z=(z_1, \dots, z_n) \in {\rm Conf}_n$ the bundle $\mathcal{G}_n$ can be trivialized as follows. Choose small open disks $D_i$ centered at $z_i$ such that $D_i \cap D_j =\emptyset$ for all $i\neq j$ and define a local chart $U^{(z)}$ around $z$ which consists of points $w=(w_1, \dots, w_n) \in {\rm Conf}_n$ such that $w_i \in D_i$. Note that the surface with boundary $\tilde{\Sigma}=\mathbb{C} \backslash \cup_i D_i$ is homotopically equivalent to $\Sigma^{(w)}=\mathbb{C}\backslash \{ w_1, \dots, w_n\}$ for all $w \in U^{(z)}$ via the inclusion $\tilde{\Sigma} \subset \Sigma^{(w)}$. The canonical isomorphism $|\mathbb{C} \pi_1(\tilde{\Sigma})| \cong |\mathbb{C}\pi_1(\Sigma^{(w)})|$  defines a trivialization of $\mathcal{G}_n$ in this chart.

The bundle $\mathcal{G}_n$ carries a canonical flat connection defined by the following family of local flat sections. 
Let $\gamma: S^1 \to \mathbb{C}$ be a closed curve. Then, the homotopy class $[\gamma]$ defines a flat section over the complement of the subset of ${\rm Conf}_n$ where $z_i \in \gamma(S^1)$ for some $i$. Moreover, it is easy to see that there is a unique local flat section passing through each point of $\mathcal{G}_n$. 

The canonical flat connection  is compatible with the Goldman-Turaev Lie bialgebra structure on the fibers. Indeed, on flat sections over $U^{(z)}$ defined by closed curves in $\tilde{\Sigma}$, the Goldman bracket and Turaev cobracket take the same value independent of $w$ (sufficiently close to $z$).

The flat bundle $\mathcal{G}_n$ admits the following description: the fundamental group of ${\rm Conf}_n$ is the pure braid group ${\rm PB}_n$. The bundle $\mathcal{G}_n$ is induced by the natural representation of ${\rm PB}_n$ on the free group $\pi_1(\Sigma)$.

\subsection{The bundle $\mathcal{H}_n$}

The natural action of ${\rm PB}_n$ on the homology $H_1(\Sigma)$ is trivial. This trivial action lifts to $|TH_1|$ and gives rise to the trivial vector bundle $\mathcal{H}_n \to {\rm Conf}_n$ with fiber $|TH_1|$.  

The bundle $\mathcal{H}_n$ carries a canonical flat connection. Recall that the Drinfeld-Kohno Lie algebra $\mathfrak{t}_n$ acts on the free Lie algebra $\mathbb{L}(a_1, \dots, a_n)$ via the identification $a_i=t_{i(n+1)} \in \mathfrak{t}_{n+1}$. Since this action is induced by the adjoint action of $\mathfrak{t}_{n+1}$ on itself, we will use the notation ${\rm ad}$. For instance, ${\rm ad}(t_{12}) a_1 = [t_{12}, t_{1(n+1)}] = 
-[t_{2(n+1)}, t_{1(n+1)}] = - [a_2, a_1]$. In particular, the central element of $\mathfrak{t}_n$
$$
c= \sum_{i<j} t_{ij}
$$
acts by the inner derivation ${\rm ad}(c) a= [a, a_1 + \dots + a_n]$. 

The canonical connection on $\mathcal{H}_n$ is (a version of) the KZ connection:
\begin{equation} \label{nabla_n}
\nabla_n=d + \frac{1}{2\pi i} \sum_{i<j} {\rm ad}(t_{ij}) d\log(z_i-z_j).
\end{equation}
It is flat because the original KZ connection is flat and because the map ${\rm ad}$ is a Lie algebra homomorphism.

\begin{Rem} \label{krv}
The image of the Lie algbera $\mathfrak{t}_n$ under the map ${\rm ad}$ is contained in the Kashiwara-Vergne Lie algebra $\mathfrak{krv}_n$ acting on $|TH_1|$ by derivations of the necklace Schedler Lie bialgebra structure \cite{genus0}. Hence, the flat connection \eqref{nabla_n} is compatible with the Lie bialgebra structure on the fibers of $\mathcal{H}_n$. 
\end{Rem}

\subsection{The holonomy is an isomorphism of flat bundles with connection}
For each $z=(z_1, \dots, z_n) \in {\rm Conf}_n$ the holonomy map $|W^{(z)}|$ maps the fiber $|\mathbb{C}\pi_1(\Sigma^{(z)})|$ of $\mathcal{G}_n$ to the fiber $|TH_1|$ of $\mathcal{H}_n$. This collection of maps for  $z \in {\rm Conf}_n$ defines a smooth bundle map $|W|: \mathcal{G}_n \to \mathcal{H}_n$.

\begin{Thm} \label{KZ}
The bundle map $|W|$ intertwines the canonical flat connection on $\mathcal{G}_n$ and the connection $\nabla_n$ on $\mathcal{H}_n$.
\end{Thm}

\begin{proof}
It is sufficient to compare the canonical connections and $\nabla_n$ on local charts. In more detail, let $\gamma: S^1 \to \tilde{\Sigma} = \sigma \backslash \cup_i D_i$  be a closed curve. Then, $[\gamma]$ defines a flat section of $\mathcal{G}_n$ over $U^{(z)}$. It is then sufficient to show that  $|W_\gamma^{(w)}|$ is a $\nabla_n$-flat section of $\mathcal{H}_n$ for $w \in U^{(z)}$.

Let $w: (-1, 1) \to {\rm Conf}_n(\mathbb{C})$ be a smooth path in $U^{(z)}$ with  $w(0)=z$ and the first derivative $w'(0)=u=(u_1, \dots, u_n)$.  Choose a starting point on the curve $\gamma$ and denote the corresponding map by $\widehat{\gamma}: [0,1] \to \tilde{\Sigma}$.

Consider the KZ connection on ${\rm Conf}_{n+1}(\mathbb{C})$ and define the map 
$\mu=w \times \widehat{\gamma}: (-1, 1) \times S^1 \to {\rm Conf}_{n+1}(\mathbb{C})$.
The pull-back of the KZ connection under $\mu$ is flat. For a given $s \in (-1,1)$ the holonomy of the induced connection along $\{ s\} \times \widehat{\gamma}$, denoted by ${\rm Hol}_\gamma(s)$, takes values in $TH_1$ and its projection to $|TH_1|$ coincides with $|W^{(w(s))}|$ (under the identification $t_{i(n+1)} = a_i$).

Since the pull-back connection is flat, the holonomies for different values of $s$ can be compared using the following formula:
$$
{\rm Hol}_\gamma(s) = {\rm Hol}_w(s \leftarrow 0) {\rm Hol}_\gamma(0) {\rm Hol}_w(0 \leftarrow s),
$$
where ${\rm Hol}_w(0 \leftarrow s) = {\rm Hol}_w(s \leftarrow 0)^{-1}$ is the holonomy of the KZ connection for $n+1$ points along the part of the path $w$ between $0$ and $s$. Differentiating this formula in $s$ at $s=0$, we obtain
$$
{\rm Hol}'_\gamma(0) = \frac{1}{2\pi  i} \, \sum_{1\leq i<j \leq n+1} \frac{u_i-u_j}{z_i-z_j} {\rm ad}(t_{ij}){\rm Hol}_\gamma(0),
$$
where $u_{n+1}=0$ and $z_{n+1} = \hat{\gamma}(0)$.
Projecting this equation to $|TH_1|$ yields
$$
(\partial_s |W^{(w(s))}|)_{s=0} = \frac{1}{2\pi i} \, \sum_{1\leq i<j \leq n} \frac{u_i-u_j}{z_i-z_j} {\rm ad}(t_{ij}) |W^{(w(0))}|.
$$
Here we have used the fact that operators ${\rm ad}(t_{i(n+1)})$ are inner derivations of $TH_1$ and they act by zero on $|TH_1|$.
The formula above shows that $|W^{(w(s))}_\gamma|$ is a flat section over $(-1,1)$. Since $w(s)$ was an arbitrary path, this  shows that $|W_\gamma^{(w)}|$ is a flat section for $\nabla_n$. 
\end{proof}

\begin{Rem}
The KZ connection on $\mathcal{H}_n$ is induced by the flat connection  on the trivial bundle $\mathcal{L}_n = {\rm Conf}_n \times \mathbb{L}(a_1, \dots, a_n)$ defined by the same formula \eqref{nabla_n}. 
 
Let $\mathfrak{g}$ be a Lie algebra (possibly $\mathbb{L}(a_1, \dots, a_n)$ itself) and consider the bundle of fiberwise Lie algebra homomorphisms of $\mathcal{L}_n$ into $\mathfrak{g}$. This is the trivial bundle ${\rm Conf}_n \times \mathfrak{g}^n$, where we identify a Lie algebra homomorphism $x: \mathbb{L}(a_1, \dots, a_n) \to \mathfrak{g}$ with the $n$-tuple $(x_1, \dots, x_n) = (x(a_1), \dots, x(a_n))$. The KZ connection induces a connection on this bundle which is computed as follows:
\begin{align*}
    ((d+ A_{\rm KZ})x)(a_k) &= dx_k + \frac{1}{2\pi i}\sum_{i<j} (t_{ij}.x)(a_k)d\log(z_i - z_j) \\
    &= dx_k + \frac{1}{2\pi i}\sum_{i<j} x(-[t_{ij},a_k])d\log(z_i - z_j) \\
    &= dx_k + \frac{1}{2\pi i}\sum_{i} x([a_i,a_k])d\log(z_i - z_k)\\
    &= dx_k + \frac{1}{2\pi i}\sum_{i} [x_i, x_k] d\log(z_i - z_k).
\end{align*}
Flat sections of this connection are solutions of the  Schlesinger equations for isomonodromic deformations:
\begin{equation}
    \label{isomo}
\frac{\partial x_k}{\partial z_i} = \frac{1}{2\pi i} \frac{[x_k, x_i]}{z_k - z_i}, \hskip 0.3cm
\frac{\partial x_k}{\partial z_k} = - \frac{1}{2\pi i} \sum_{i \neq k} \frac{[x_k, x_l]}{z_k-z_l}.
\end{equation}
For $\mathfrak{g} = \mathbb{L}(a_1, \dots, a_n)$, let $x_k = x_k(a_1,\dots,a_n, z_1, \dots, z_n) = x_k(z)$ be local solutions to \eqref{isomo}, such that they generate $\mathbb{L}(a_1, \dots, a_n)$ in each fiber. We call $(x_1, \dots, x_n)$ isomonodromic coordinates. 

Let $\gamma \subset \Sigma^{(z)}$ be a closed curve. It defines a local flat section $[\gamma]$ of $\mathcal{G}_n$ and  by Theorem \ref{KZ} it gives rise to a local flat section $|W_\gamma|$ of $\mathcal{H}_n$. By the considerations above, $|W_\gamma|$ is represented by a constant  (independent of $z_1, \dots, z_n$) function of the isomonodromic coordinates $(x_1, \dots, x_n)$.

\end{Rem}


%
%


\subsection{The moduli space of curves}
Recall that the moduli space of genus zero curves with $n+1$ marked points $\mathcal{M}_{0, n+1}$ can be defined as the quotient of the configuration space of $n+1$ points on $\mathbb{C}P^1$ modulo the natural ${\rm PSL}_2$-action by M\"obius transformations:
$$
\mathcal{M}_{0, n+1} = {\rm Conf}_{n+1}(\mathbb{C}P^1)/{\rm PSL}_2.
$$
Equivalently, one can fix one of the points to be at infinity of $\mathbb{C}P^1$ and consider:
$$
\mathcal{M}_{0, n+1} \cong {\rm Conf}_n(\mathbb{C})/\Gamma, 
$$
where  $\Gamma$ is the group of translations and dilations
$$
\Gamma= \{ z \mapsto az+b; a \in \mathbb{C}^*, b \in \mathbb{C}\}.
$$

The action of the group $\Gamma$ lifts to the bundles $\mathcal{G}_n$ and $\mathcal{H}_n$ as follows. The bundle $\mathcal{H}_n = {\rm Conf}_n \times |TH_1|$ is trivial and one extends the action of $\Gamma$ on ${\rm Conf}_n$ by the trivial action on the fibers. For the bundle $\mathcal{G}_n$, the group $\Gamma$ acts by diffeomorphisms $\Sigma^{(z)} \to \Sigma^{(az+b)}$, where $(az+b)=(az_1 +b, \dots, az_n+b)$. The action of $\Gamma$ on $\mathcal{G}_n$ is the induced action on the homotopy classes of curves 
$$
[\gamma(s)] \mapsto [\gamma_{a,b}(s)=a\gamma(s) +b].
$$


In conclusion, $\mathcal{G}_n$ and $\mathcal{H}_n$ become equivariant vector bundles under the action of $\Gamma$ and they give rise to vector bundles over $\mathcal{M}_{0, n+1}$ (that we again denote by $\mathcal{G}_n$ and $\mathcal{H}_n$).

The canonical flat connection on $\mathcal{G}_n$ descends from ${\rm Conf}_n$ to $\mathcal{M}_{0, n+1}$ because the action of $\Gamma$ maps flat sections to flat sections. It turns out that the KZ connection on $\mathcal{H}_n$ also descends to the moduli space:

\begin{Prop}
The flat connection 
$$
\nabla_n=d + \frac{1}{2\pi i} \sum_{i<j} {\rm ad}(t_{ij}) d\log(z_i-z_j)
$$
descends to a flat connection on $\mathcal{M}_{0, n+1}$.
\end{Prop}

\begin{proof}
It is obvious that the connection 1-form is basic under the action of the group of translations $z \mapsto z+b$ and invariant under the action of dilations $z\mapsto \lambda z$. In order to see that it is horizontal under the action of dilations, define the Euler vector field $v=\sum_{i=1}^n z_i \partial_i$. Contracting it with the connection 1-form yields
$$
\iota(v) \sum_{i<j} {\rm ad}(t_{ij}) d\log(z_i-z_j) = \sum_{i<j} {\rm ad}(t_{ij}) = {\rm ad}(c) = 0,
$$
where we have used the fact that $c$ acts by zero on $|TH_1|$.
\end{proof}

Finally, the holonomy map descends to a bundle map over the moduli space as well:

\begin{Prop}
The holonomy map $|W|$ descends to $\mathcal{M}_{0, n+1}$.
\end{Prop}

\begin{proof}
The group $\Gamma$  acts on the complex plane by diffeomorphisms and the connection 

$$
A^{(z_1, \dots, z_n)} = \frac{1}{2\pi i} \sum_{i=1}^n a_i d\log(z-z_i)
$$
is mapped to $A^{(az +b)}$. Hence,  the holonomy map $W^{(z)}$ evaluated on a curve $\gamma: [0,1] \to \Sigma^{(z)}$ coincides with the holonomy map
$W^{(az+b)}$ evaluated on the curve $\gamma_{a,b}: [0,1] \to \Sigma^{(az+b)}$. Applying projection $TH_1 \to |TH_1|$, we conclude that $|W|$ intertwines the actions of $\Gamma$ on $\mathcal{G}_n$ and $\mathcal{H}_n$. 

\end{proof}

We can summarize the statements above as follows:
\begin{Thm} \label{moduli_final}
The flat bundles with connections $\mathcal{G}_n$ and $\mathcal{H}_n$ and the bundle map $|W|$ over the configuration space ${\rm Conf}_n$ descend to the moduli space of genus zero curves $\mathcal{M}_{0, n+1}$.
\end{Thm}

\begin{Rem}
By Remark \ref{krv}, the pure braid group ${\rm PB}_n$ (the fundamental group of the configuration space ${\rm Conf}_n$) acts on $|TH_1|$ by automorphisms of the necklace Schedler Lie bialgebra structure. Theorem \ref{moduli_final} shows that this action only depends on the complex structure on the surface (that is, on the point of $\mathcal{M}_{0, n+1}$).

\end{Rem}

%
%
%
%
%

\section{Relation to the KZ associator and the Kashiwara-Vergne problem}

In this section we establish the relation between the map $|W|: |\mathbb{C}\pi_1| \to TH_1$ and the solution of the Kashiwara-Vergne problem defined by the KZ associator.

Recall the following statements:

\begin{Thm}[Theorem 5 in \cite{AEM}]
For every Drinfeld associator $\Phi(x,y)$ the automorphism $F_\Phi$ of the free Lie algebra $\mathcal{L}(x,y)$ defined by formulas
\begin{align*}
x & \mapsto \Phi(x, -x-y) x \Phi(x, -x-y)^{-1}, \\
y & \mapsto e^{-(x+y)/2} \Phi(y, -x-y) y \Phi(y, -x-y)^{-1} e^{(x+y)/2}
\end{align*}
is a solution of the Kashiwara-Vergne problem.
\end{Thm}

For the second statement we need the following notation: let $\Sigma= \mathbb{C}\backslash \{ 0, 1\}$, choose a base point $p \in [0,1]$ and two generators $\gamma_0, \gamma_1 \in \pi_1$ as shown on Fig. \ref{fig1}.

\begin{figure}
    \centering
    
\begin{tikzpicture} 
\fill[black] (0,0) circle (0.05) node[above left]{$0$}; 
\fill[black] (2,0) circle (0.05) node[above left]{$1$}; 
\draw (2,0)  node[above left] {$1$};
\draw (0.4,0) -- (0.6, 0) ;
\draw (0.5,-0.1) -- (0.5, 0.1) ;
\draw (0.5,0) .. controls (1.5,1) and (3,1) .. (3, 0) node[pos=0, above left] {$p$} node[pos=.3, above] {$\gamma_0$};
\draw[->] (0.5,0) .. controls (1.5,-1) and (3,-1) .. (3, 0);

\draw[->] (0.5,0) .. controls (0.5,1) and (-0.5,1) .. (-0.5, 0) node[pos=.3, above] {$\gamma_1$};
\draw (0.5,0) .. controls (0.5,-1) and (-0.5,-1) .. (-0.5, 0);
\end{tikzpicture}

    \caption{Definition of $\gamma_0$ and $\gamma_1$}
    \label{fig1}
\end{figure}
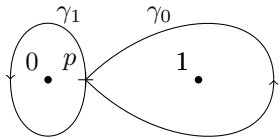

\begin{Thm}[Theorem 7.6 in \cite{genus0}]
Let $F$ be an automorphism of the free Lie algebra $\mathcal{L}(x,y)$ which solves the Kashiwara-Vergne problem. Then, the map $\rho_F: \mathbb{C} \pi_1 \to TH_1$ defined on generators by
$$
\rho_F(\gamma_0) = \exp(F(x)), \hskip 0.3cm \rho_F(\gamma_1)= \exp(F(y))
$$
induces an isomorphism of Lie bialgebras $|\mathbb{C} \pi_1| \to |TH_1|$.
\end{Thm}

We will now  prove the following proposition:

\begin{Prop}
For $\Sigma = \mathbb{C} \backslash \{ 0, 1\}$ the Lie bialgebra isomorphism $|W|$ coincides with the one induced by the solution of the Kashiwara-Vergne problem corresponding to the KZ associator $\Phi_{\rm KZ}(x,y)$.
\end{Prop}

Let $n=2$ and $z_1=0, z_2=1$. Denote $x=a_1, y=a_2$. The connection $d+A$ acquires the form
$$
d+A = d + \frac{1}{2\pi i} \left(x \, d \, \log(z) + y \, d\, \log(z-1)\right).
$$
Following Drinfeld \cite{Drinfeld}, define fundamental solutions of the equation $(d+A)\Psi=0$ with asymptotics
$$
\Psi_0(z) = (1 + O(z))z^{\frac{x}{2\pi i}}, \hskip 0.3cm
\Psi_1(z)= (1 + O(z))(z-1)^{\frac{y}{2\pi i} }.
$$
Choose the base point $p \in [0,1]$ and a basis in $\pi_1$ represented by the curves $\gamma_0$ which surrounds $0$ and $\gamma_1$ which surrounds $1$. Define the representation $\rho_{\rm KZ}: \pi_1 \to G_2$ as follows:
$$
\rho_{\rm KZ}(\gamma_0) = e^x, \hskip 0.3cm
\rho_{\rm KZ}(\gamma_1) = \Phi_{\rm KZ}(x,y)^{-1} e^y \Phi_{\rm KZ}(x,y),
$$
where $\Phi_{\rm KZ}$ is the KZ associator.

\begin{Lem}
The representations $\rho_{\rm KZ}$ and $W$ are equivalent under conjugation by $\Psi_0(p)$. The corresponding maps $|\mathbb{C}\pi_1| \to TH_1$ coincide.
\end{Lem}

\begin{proof}
By definition of the holonomy,
$$
\Psi_0(\gamma_0 \cdot p) = {\rm Hol}_{\gamma_0} \Psi_0(p), \hskip 0.3cm
\Psi_1(\gamma_1 \cdot p) = {\rm Hol}_{\gamma_1} \Psi_1(p).
$$
In combination with 
$$
\Psi_0(\gamma_0 \cdot p) = \Psi_0(p) e^x, \hskip 0.3cm
\Psi_1(\gamma_1 \cdot p) = \Psi_1(p) e^y, \hskip 0.3cm
\Psi_1(p) = \Psi_0(p) \Phi_{\rm KZ}(x,y)
$$
we obtain
$$
\rho_{\rm KZ}(\gamma_0) = e^x = \Psi_0(p)^{-1} \Psi_0(\gamma_0 \cdot p) = \Psi_0(p)^{-1} {\rm Hol}_{\gamma_0} \Psi_0(p)
$$
and
\begin{align*}
\rho_{\rm KZ}(\gamma_1) & =  
\Phi_{\rm KZ}^{-1}e^y \Phi_{\rm KZ}
= \Phi_{\rm KZ}^{-1}\Psi_1(p)^{-1} \Psi_1(\gamma_1 \cdot p) \Phi_{\rm KZ} \\
& = (\Psi_1(p) \Phi_{\rm KZ})^{-1} {\rm Hol}_{\gamma_1} (\Psi_1(p) \Phi_{\rm KZ}) =
\Psi_0(p)^{-1} {\rm Hol}_{\gamma_1} \Psi_0(p),
\end{align*}
as required. 

Since the two representations of $\pi_1$ are equivalent, they descend to the same map on $|\mathbb{C}\pi_1|$.

\end{proof}

\begin{proof}[Proof of Proposition]
To complete the proof of Proposition, we remark that the representation $\rho_{\rm KZ}$ is equivalent to the representation
\begin{align*}
\rho(\gamma_0) & =  \Phi(x, -x-y) e^x \Phi(x, -x-y)^{-1}, \\
\rho(\gamma_1) & =  e^{-(x+y)/2} \Phi(y, -x-y) e^y \Phi(y, -x-y)^{-1} e^{(x+y)/2}. 
\end{align*}

The equivalence is given by conjugation with $e^{x/2} \Phi(-x-y, x)$ which corresponds to moving the base point from the neighborhood of $0$ to the neighborhood of $\infty$.

The map $|\mathbb{C}\pi_1| \to | TH_1|$ induced by $\rho$ coincides with the one induced by the solution of the Kashiwara-Vergne problem corresponding to the KZ associator $\Phi_{\rm KZ}(x,y)$.
\end{proof}

\begin{Rem}
G. Massuyeau explained to us that for $n=2$ the Lie bialgebra isomorphism $|\mathbb{C} \pi_1|\to |TH_1|$ constructed in \cite{mas} using the Kontsevich integral coincides with the map described above if one chooses the associator $\Phi= \Phi_{\rm KZ}$.
\end{Rem}

{\bf Acknowledgements.} This work was inspired by  discussions with R. Hain. We are grateful to P. Boalch, N. Kawazumi, M. Kontsevich, Y. Kuno, G. Massuyeau, P. Safronov, P. Severa and X. Xu for the interest in our results and for useful comments. Our work was supported in part by the ERC project MODFLAT, by the grants 165666 and 159581 of the Swiss National Science Foundation (SNSF) and by the NCCR SwissMAP.

\end{document}